\documentclass
[
11pt]
{amsart}%

\usepackage{hyperref}
\usepackage{cite}
\usepackage{amssymb}
\usepackage{amsmath}
\usepackage{amsthm}
\usepackage{esint}
\usepackage{amsfonts}
\usepackage{graphicx}
\usepackage{bbm}
\usepackage{xcolor}
\usepackage[toc,page]{appendix}
\usepackage{subfigure}
\usepackage{mathtools}
\usepackage{mathrsfs}
\usepackage{bm} 
\usepackage[normalem]{ulem}
\usepackage{comment}
\usepackage{tikz, tikz-cd}
\usetikzlibrary{matrix,arrows,decorations.pathmorphing}

\newtheorem{theorem}{Theorem}[section]

\newtheorem{notation}[theorem]{Notation}
\newtheorem{proposition}[theorem]{Proposition}

\newtheorem{lemma}[theorem]{Lemma}
\newtheorem{remark}[theorem]{Remark}
\newtheorem{example}[theorem]{Example}

\theoremstyle{definition}
\newtheorem{definition}[theorem]{Definition}

\numberwithin{equation}{section}

\begin{document}

\title[Infinite-dimensional reduced Heisenberg]{Logarithmic Sobolev inequalities on infinite-dimensional reduced Heisenberg groups}

\author[Maria Gordina]{Maria Gordina{$^{\dag}$}}
\thanks{\footnotemark {$\dag$} Research was supported in part by NSF Grant DMS-2246549.}
\address{Department of Mathematics\\
University of Connecticut\\
Storrs, CT 06269,  U.S.A.}
\email{maria.gordina@uconn.edu}

\author[Liangbing Luo]{Liangbing Luo}
\address{
Department of Mathematics and Statistics\\
Queen's University\\
Kingston, ON, Canada
K7L 3N6}
\email{liangbing.luo@queensu.ca}

\maketitle

\begin{abstract}
We construct a family of infinite-dimensional reduced Heisenberg groups which can be viewed as infinite-dimensional homogeneous spaces. Such a space is an analogue of finite-dimensional reduced Heisenberg groups in infinite dimensions. We study properties of the hypoelliptic heat kernel measure on this space, including hypoelliptic logarithmic Sobolev inequalities there.
\end{abstract}

\tableofcontents

\section{Introduction}

The logarithmic Sobolev inequality has been first introduced and studied by L.~Gross in \cite{Gross1975c} on a Euclidean space with respect to the Gaussian measure. This inequality is fundamental to analysis on infinite-dimensional spaces, as the logarithmic Sobolev constant (LSI constant) has been proven to be dimension-free thus allowing it to be used in infinite dimensions. Ever since then, it found many applications in both finite and infinite dimensional settings. 

In the setting of finite-dimensional Riemannian manifolds, the work by Bakry-\'{E}mery-Ledoux et al shows that the logarithmic Sobolev constant for the heat kernel measure only depends on the Ricci lower bound while it is independent of the dimension. More generally, one can prove such inequalities for certain class of uniformly elliptic operators. However, in infinite-dimensional settings such an approach is not easily available as many of these techniques are inherently finite-dimensional.  While some functional inequalities can be proven on direct and inductive limits spaces as in \cite{DriverGordina2009}, giving a unified approach to proving logarithmic Sobolev inequalities on non-linear infinite-dimensional spaces seems to be elusive. In addition to lacking PDEs' techniques, one has to choose a reference measure in the absence of the Lebesgue measure, and address subtle topological issues.  There have been a number of results on some such spaces including path spaces, e.~g. \cite{Gross1975c, CarlenStroock1986, Zegarlinski1990, Gross1991, Gross1992, DriverLohrenz1996, Hsu1997, DriverGordina2008, Melcher2009, GordinaLuo2022}.

In the present paper, we construct a family of infinite-dimensional \emph{reduced} Heisenberg groups modelled on an abstract Wiener space. This group can be seen as an infinite-dimensional analogue of the classical finite-dimensional reduced Heisenberg considered in \cite[Section 3]{Thangavelu1991}. The construction of such a group is closely related to the infinite-dimensional Heisenberg group first introduced in \cite{DriverGordina2008}. As there is no natural Haar measure in such a setting, our analysis is done with respect to the hypoelliptic heat kernel measure.

More precisely, the Lie algebra of the infinite-dimensional reduced Heisenberg group satisfies an a Lie bracket-generating (H\"{o}rmander's) condition. This induces a \emph{sub-Laplacian} on such a space, which can be seen as an analogue of the sub-Laplacian operator in finite-dimensional sub-Riemannian geometry. While hypoellipticity in infinite dimensions cannot be defined using the original H\"{o}rmander's approach in \cite{Hormander1967a}, this is another infinite-dimensional setting where the corresponding differential operator should be viewed as hypoelliptic. Our main goal is  to study the logarithmic Sobolev inequality with respect to the hypoelliptic heat kernel measure corresponding to this operator.

In this paper, our approach to study logarithmic Sobolev inequalities is quite different from those in the existing mathematical literature on  infinite-dimensional spaces, such as finite-dimensional approximations and tensorization in \cite{DriverGordina2008, Melcher2009, GordinaLuo2022}. Our method relies on a group action on the underlying space inducing a map between Dirichlet spaces. This is the first time that such a method has been applied in infinite-dimensional settings to study the logarithmic Sobolev inequality.

The motivation for our approach comes from L.~Gross' work in \cite{Gross1992} for finite-dimensional Riemannian manifolds, where (elliptic) logarithmic Sobolev inequalities on quotient spaces have been studied. Later, this approach in the finite-dimensional setting has been explored more for the heat kernel analysis on homogeneous spaces. For example, \cite{DriverGrossSaloff-Coste2010} used it to prove that the Taylor map on complex manifolds is unitary, and \cite{GordinaLuo2024} first applied it to the sub-Riemannian setting to show  logarithmic Sobolev constants are dimension-independent on homogeneous spaces. While the Taylor map on infinite-dimensional Heisenberg groups have been studied in \cite{DriverGordina2010, GordinaMelcher2013}, logarithmic Sobolev inequalities for hypoelliptic heat kernel measures have not been considered there.  We start by showing that logarithmic Sobolev inequalities are invariant under quasi-homeomorphisms between quasi-regular Dirichlet spaces.  We hope that such a method can be applied to more infinite-dimensional  settings or to prove other properties such as hypercontractivity and different functional inequalities, but in the paper we focus on the study of a logarithmic Sobolev inequality on one concrete class of infinite-dimensional spaces here. 

The key idea is that the group action we consider preserves cylinder function spaces, sub-Laplacians and horizontal gradients, and thus give us an example of such a quasi-homeomorphism. The hypoelliptic heat kernel measure on the quotient space is the pushforward under the quotient map of the hypoelliptic heat kernel measure on the Heisenberg group. Finally, we show that the LSI constant does not increase under such a map. 

Our paper is organized as follows. First, in Section~\ref{sec.LSI.Quasi-homeomorphism} we prove a general result on logarithmic Sobolev inequalities on quasi-homeomorphic quasi-regular Dirichlet forms. In Section~\ref{sec.ReducedHeisenberg} we construct the infinite-dimensional reduced Heisenberg group. Next in Section~\ref{sec.Properties}, we explore its basic properties, including those of the corresponding sub-Laplacian, heat kernel measure, Dirichlet form. Finally, in Section~\ref{sec.LSI.Reduced}, we deduce a logarithmic Sobolev inequality with respect to the hypoelliptic heat kernel measure and we discuss the LSI constant.

\section{Logarithmic Sobolev inequalities under quasi-homeomorphisms} \label{sec.LSI.Quasi-homeomorphism}

We start by formulating a general result about logarithmic Sobolev inequalities  under quasi-homeomorphisms between  quasi-regular Dirichlet spaces. For simplicity, we only formulate a theorem and give a proof in such a setting here. For more details related to quasi-regular Dirichlet spaces, we refer to \cite{ChenZQMaZMRockner1994, Fukushima1999}.

Let $(E,\mu,\mathcal{E},\mathcal{D}(\mathcal{E}))$ be a quasi-regular Dirichlet space, where $E$ is a Hausdorff topological space with a countable base, $\mu$ is a $\sigma$-finite positive Borel measure on $E$ and $\mathcal{E}$ a quasi-regular Dirichlet form with its domain $\mathcal{D}(\mathcal{E}) \subseteq L^2(E,d\mu)$.  

\begin{notation}
The $\mathcal{E}$-norm is given by
\begin{align*}
\Vert f \Vert_{\mathcal{E}}^2:=\Vert f\Vert^2_{L^2\left(E,d\mu\right)}+\mathcal{E}(f,f)
\end{align*}
for any $f\in \mathcal{D}(\mathcal{E})$. Given a closed subset $F \subseteq E$, we denote
\begin{align*}
\mathcal{D}(\mathcal{E})_{F}=\{f\in \mathcal{D}(\mathcal{E}):f=0 \, \mu-a.e. \, \text{on} \, F^c\}.
\end{align*}
\end{notation}

By the definition of quasi-regular Dirichlet forms, there exists an \emph{$\mathcal{E}$-nest}, that is, there exists an increasing sequence $\{F_k\}_{k=1}^{\infty}$ of closed subsets of $E$ such that $\bigcup\limits_{k=1}^{\infty} \mathcal{D}(\mathcal{E})_{F_k}$ is dense in $\mathcal{D}(\mathcal{E})$ with respect to the $\|\cdot\|_{\mathcal{E}}$-norm.

\begin{definition}
A quasi-regular Dirichlet space $(E_1,\mu_1,\mathcal{E}_1,\mathcal{D}(\mathcal{E}_1))$ is  \emph{quasi-homeomorphic} to another quasi-regular Dirichlet space $(E_2,\mu_2,\mathcal{E}_2,\mathcal{D}(\mathcal{E}_2))$ if there exists a map $\varphi: E_1 \longrightarrow E_2$ such that
\begin{enumerate}
    \item the pushforward of $\mu_1$ to $E_2$ is $\mu_2$, that is, $\mu_2=\varphi_{\#}\mu_1$,
    \item for each $k$, the restriction of map $\varphi$ to $F_1^k$ is a homeomorphism from $F^k_1$ onto $F^k_2$,
    \item for any $f,h\in \mathcal{D}(\mathcal{E}_2)$, we have $f\circ \varphi, h\circ \varphi \in \mathcal{D}(\mathcal{E}_1)$ and $\mathcal{E}_2(f,h)=\mathcal{E}_1(f\circ \varphi, h\circ \varphi)$.
\end{enumerate}
Such a map $\varphi$ is said to be a \emph{quasi-homeomorphism} from $(E_1,\mu_1,\mathcal{E}_1,\mathcal{D}(\mathcal{E}_1))$ to $(E_2,\mu_2,\mathcal{E}_2,\mathcal{D}(\mathcal{E}_2))$.
\end{definition}

\begin{notation}
We say that a quasi-regular Dirichlet space $(E,\mu,\mathcal{E},\mathcal{D}(\mathcal{E}))$ satisfies a logarithmic Sobolev inequality with the constant $C\left(E,\mu,\mathcal{E},\mathcal{D}(\mathcal{E})\right)$ if
\begin{align} \label{LSI}
\int_{E}f^2\log f^2d\mu-\left(\int_{E}f^2 d\mu\right)\log\left(\int_{E}f^2d\mu\right) \leqslant C\left(E,\mu,\mathcal{E},\mathcal{D}(\mathcal{E})\right) \mathcal{E}(f,f)
\end{align}
for any $f\in \mathcal{D}\left(\mathcal{E}\right)$. In such a case we also say that $LSI_C(E,\mu,\mathcal{E},\mathcal{D}(\mathcal{E}))$ holds. 
\end{notation}

\begin{theorem} \label{thm.LSI.Quasi-homeomorphism}
Given two quasi-regular Dirichlet spaces $(E_i,\mu_i,\mathcal{E}_i,\mathcal{D}(\mathcal{E}_i))$, suppose $\varphi: E_1 \longrightarrow E_2$ is a quasi-homeomorphism between these two Dirichlet spaces. If $LSI_C(E_1,\mu_1,\mathcal{E}_1,\mathcal{D}(\mathcal{E}_1))$ holds on $E_1$, then $LSI_C(E_2, \mu_2, \mathcal{E}_2, \mathcal{D}(\mathcal{E}_2))$ holds on $E_2$. Moreover, the constant $C(E_2,\mu_2,\mathcal{E}_2,\mathcal{D}(\mathcal{E}_2))$ can be chosen to be
\begin{align}
C(E_2,\mu_2,\mathcal{E}_2,\mathcal{D}(\mathcal{E}_2))=C(E_1,\mu_1,\mathcal{E}_1,\mathcal{D}(\mathcal{E}_1)).
\end{align}
\end{theorem}

\begin{proof}
For any $f\in \mathcal{D}(\mathcal{E}_2)$, we have $f\circ \varphi\in \mathcal{D}(\mathcal{E}_1)$. Using the change of variables $\varphi(x_1)\mapsto x_2 \in E_2$ in 
\begin{align} 
& \int_{E_1}(f\circ\varphi)^2\log (f\circ\varphi)^2d\mu_1(x_1)
\notag
\\
&-\left(\int_{E_1}(f\circ\varphi)^2 d\mu_1(x_1)\right)\log\left(\int_{E_1}(f\circ\varphi)^2d\mu_1(x_1)\right) \\
&
\leqslant C(E_1,\mu_1,\mathcal{E}_1,\mathcal{D}(\mathcal{E}_1)) \mathcal{E}_1(f\circ\varphi,f\circ\varphi),
\notag
\end{align}
we see that 
\begin{align*} 
& \int_{E_2}f^2\log f^2d\mu_2(x_2)-\left(\int_{E_2}f^2 d\mu_2(x_2)\right)\log\left(\int_{E_2}f^2d\mu_2(x_2)\right) 
\\
& \leqslant C(E_2,\mu_2,\mathcal{E}_2,\mathcal{D}(\mathcal{E}_2)) \mathcal{E}_2(f,f),
\end{align*}
and we can choose $C(E_2,\mu_2,\mathcal{E}_2,\mathcal{D}(\mathcal{E}_2))=C(E_1,\mu_1,\mathcal{E}_1,\mathcal{D}(\mathcal{E}_1))$.
\end{proof}

\begin{remark}
We see that an LSI constant can be chosen to be the same for quasi-homeomorphic Dirichlet spaces. This gives a way to control the LSI constant under quasi-homeomrphisms.
\end{remark}

\begin{remark}
In Section~\ref{sec.ReducedHeisenberg} and \ref{sec.Properties}, we show that the group action of the normal subgroup $\Gamma$ on an infinite-dimensional Heisenberg group $G$ induces a quasi-homeomorphisms between the Dirichlet spaces on $G$ and the corresponding reduced Heisenberg group. 
\end{remark}

\section{Infinite-dimensional reduced Heisenberg groups} \label{sec.ReducedHeisenberg}

We start by constructing a family of infinite-dimensional \emph{reduced} Heisenberg groups, which are infinite-dimensional Lie groups modelled on an abstract Wiener space. Our construction is closely related to the infinite-dimensional Heisenberg group first introduced in \cite{DriverGordina2008}. In this section, we first describe our construction and then provide some examples.

\subsection{Abstract Wiener spaces}\label{s.wiener}

The concept of an abstract Wiener space goes back to L.~Gross in \cite{Gross1967} which allowed us to understand the structure of Gaussian measures on infinite-dimensional spaces. To keep the exposition self-contained, we start by summarizing well-known properties of Gaussian measures and abstract Wiener spaces needed later.  These results as well as more details on abstract Wiener spaces may be found in \cite{BogachevGaussianMeasures, KuoLNM1975}.

Suppose that $W$ is a real separable Banach space and $\mathcal{B}_{W}$ is the Borel $\sigma$-algebra on $W$.

\begin{definition}
\label{d.2.1}
A measure $\mu$ on $(W,\mathcal{B}_{W})$ is called a (mean zero,
non-degenerate) {\it Gaussian measure} provided that its characteristic
functional is given by
\begin{equation}
\label{e.gauss}
\hat{\mu}(u) := \int_W e^{iu(x)} d\mu(x)
	= e^{-\frac{1}{2}q(u,u)}, \qquad \text{ for all } u \in W^*,
\end{equation}
for $q=q_\mu:W^*\times W^*\rightarrow\mathbb{R}$ a symmetric, positive definite quadratic form.
That is, $q$ is a real inner product on $W^*$.
\end{definition}

A proof of the following standard theorem may be found for example in  \cite[Appendix
A]{DriverGordina2008} and \cite[Lemma 3.2]{BaudoinGordinaMelcher2013}.

\begin{theorem}
\label{t.2.3}
Let $\mu$ be a Gaussian measure on a real separable Banach space $W$.
For $p\in[1,\infty)$, let
\begin{equation}
\label{e.2.2}
C_p :=\int_W \|w\| _{W}^{p} \,d\mu(w).
\end{equation}
For $w\in W$, let
\[
\|w\|_H := \sup\limits_{u\in W^*\setminus\{0\}}\frac{|u(w)|}{\sqrt{q(u,u)}}
\]
and define the {\em Cameron-Martin subspace} $H\subset W$ by
\[ H := \{h\in W : \|h\|_H < \infty\}. \]
Then
\begin{enumerate}
\item \label{i.1}
For all  $p\in[1,\infty)$, $C_p<\infty$.

\item $H$ is a dense subspace of $W$.

\item There exists a unique inner product $\langle\cdot,\cdot\rangle_H$
on $H$ such that $\|h\|_H^2 = \langle h,h\rangle_H$ for all $h\in H$, and
$H$ is a separable Hilbert space with respect to this inner product.

\item \label{i.3}
For any $h\in H$,
$\|h\|_W \le \sqrt{C_2} \|h\|_H$.

\item \label{i.5}
If $\{e_j\}_{j=1}^\infty$ is an orthonormal basis for $H$, then for any
$u,v\in H^*$
\[ q(u,v) = \langle u,v\rangle_{H^*} = \sum_{j=1}^\infty u(e_j)v(e_j).
\]
\item \label{l.q}
If $u,v\in W^*$, then
\[
\int_W u(w)v(w)\,d\mu(w) = q(u,v).
 \]
\end{enumerate}
\end{theorem}

\subsection{Reduced Heisenberg groups}

We construct a family of infinite-dimensional reduced Heisenberg groups and their Cameron-Martin subgroups.

We first revisit the definition of the infinite-dimensional Heisenberg-like groups. Here we only consider the case when the center is one-dimensional.  For more details and explanation, we refer to \cite{DriverGordina2008, GordinaMelcher2013, BaudoinGordinaMelcher2013, DriverEldredgeMelcher2016, Gordina2017, GordinaLuo2022} etc.

Suppose $\omega: W \times W \rightarrow \mathbb{R}$ is a continuous non-degenerate skew-symmetric bilinear (symplectic) form on $W$ and $\operatorname{Span}\{\omega(u,v):u,v\in H\}=\mathbb{R}$.

\begin{definition} \label{df.InfiniteDimensionalHeisenbergGroup}
An \emph{infinite-dimensional Heisenberg-like group} $G$ is the set $W\times\mathbb{R}$ equipped with the group law given by
\begin{equation}
\label{eqn.GroupLaw.Heisenberg}
(w_1,c_1)\cdot(w_2,c_2) = \left( w_1 + w_2, c_1 + c_2 +
    \frac{1}{2}\omega(w_1,w_2)\right).
\end{equation}
for any $g_i=(w_i,c_i) \in W\times \mathbb{R}$. The \emph{Cameron-Martin subgroup} $G_{CM}$ is the subset $H\times \mathbb{R}$ equipped with the same group law as \eqref{eqn.GroupLaw.Heisenberg} restricted to $H\times \mathbb{R}$.
\end{definition}

\begin{remark}
The group $G$ is modelled on a Banach space with respect to the norm
\[
|(w,c)|_G := \sqrt{\|w\|_W^2+\vert c\vert^2}.
\]
The Cameron-Martin subgroup $G_{CM}$ is modelled on a Hilbert space with respect to the inner product
\[
\langle (A,a),(A^{\prime},a^{\prime})\rangle_{G_{CM}} :=\langle A,A^{\prime}\rangle_H+aa^{\prime}
\]
and the associated Hilbertian norm is given by
\[
|(A,a)|_{G_{CM}}:=\sqrt{\|A\|_H^2+\vert a\vert^2}.
\]
In this case, 
both $G$ and $G_{CM}$ are infinite-dimensional Lie groups, with the same identity element being $e=\left( 0, 0 \right)$. 
\end{remark}

Next we  construct infinite-dimensional \emph{reduced} Heisenberg groups and corresponding Cameron-Martin subgroups. Consider the central subgroup $\Gamma$ of $G$ given by
\begin{align*}
\Gamma:=\{(\mathbf{0},2k\pi):k\in \mathbb{Z}\}.
\end{align*}

\begin{definition}
We define the \emph{infinite-dimensional reduced Heisenberg group} $\widetilde{G}$ as the quotient space $G \backslash \Gamma$. We define the \emph{Cameron-Martin subgroup} $\widetilde{G}_{CM}$ of $\widetilde{G}$ as the quotient space $G_{CM} \backslash \Gamma$. 
\end{definition}

\begin{remark}
Note that both $\widetilde{G}$ and $\widetilde{G}_{CM}$ are indeed topological groups, as $\Gamma$ is a normal subgroup of both $G$ and $G_{CM}$. 
\end{remark}

For $g_i=(w_i,c_i)$, the group operation is still given by
\begin{equation*}
(w_1,c_1)\cdot(w_2,c_2) = \left( w_1 + w_2, c_1 + c_2 +
    \frac{1}{2}\omega(w_1,w_2)\right).
\end{equation*}
where $c_1 + c_2 +\frac{1}{2}\omega(w_1,w_2)$ is taken modulo $2\pi$. 

The center of the
reduced Heisenberg group is then compact and can be identified with
the one-dimensional torus  $[0,2\pi) \cong \mathbb{T}^1$. In this way, we can have the identification $\widetilde{G}\cong W \times \mathbb{T}^1$ and $\widetilde{G}_{CM} \cong H \times \mathbb{T}^1$. If we use the angle coordinate on $\mathbb{T}^1$, an equivalent definition of the infinite-dimensional reduced Heisenberg group $\widetilde{G}$ can be given as the topological space $W \times \mathbb{T}^1$ equipped with the group law given by 
\begin{align} \label{eqn.GroupLaw.Reduced}
(w_1,e^{c_1 i})\cdot(w_2,e^{c_2 i}) = \left( w_1 + w_2, e^{(c_1 + c_2 +
    \frac{1}{2}\omega(w_1,w_2))i}\right).
\end{align}
for any $(w_1,e^{c_1 i}),(w_2,e^{c_2 i}) \in \widetilde{G}$ with $w_1,w_2\in W$ and $c_1,c_2\in \mathbb{R}$. Similarly the \emph{Cameron-Martin subgroup} $\widetilde{G}_{CM}$ is the subset $H\times \mathbb{T}^1$ equipped with the same group law as \eqref{eqn.GroupLaw.Reduced} restricted to $H\times \mathbb{T}^1$.

\begin{remark}
We see that $\widetilde{G}$ is a smooth Banach manifold while $\widetilde{G}_{CM}$ is a smooth Hilbert manifold.
\end{remark}

Next, we discuss the correspondence between  these infinite-dimensional Lie groups and  Lie algebras. Note that the tangent space $T_{\widetilde{e}}\widetilde{G}$ at $\widetilde{e}\in \widetilde{G}$ can be identified with the space  $W\times \mathbb{R}$. It is a Lie algebra (denoted by  $\widetilde{\mathfrak{g}} \cong T_{\widetilde{e}}\widetilde{G}$) corresponding to $G$ when equipped with the Lie bracket given by
\begin{equation}\label{e.LieBracket.Reduced}
\left[
\left( \mathbf{a}_{1}, c_{1} \right), \left( \mathbf{a}_{2}, c_{2} \right)  \right]_{\widetilde{\mathfrak{g}}} = \left(0,
\omega\left( \mathbf{a}_{1}, \mathbf{a}_{2} \right)  \right).
\end{equation}
Also, the tangent space $T_{\widetilde{e}}\widetilde{G}_{CM}$ at $\widetilde{e}\in \widetilde{G}_{CM}$ can be identified with the space  $H\times \mathbb{R}$. It is a Lie algebra (denoted by  $\widetilde{\mathfrak{g}}_{CM}$) corresponding to $\widetilde{G}_{CM}$ when equipped with the same Lie bracket as \eqref{e.LieBracket.Reduced} restricted to $H\times \mathbb{R}$. Moreover, note that the surjectivity of $\omega$ implies 
\[
[W \times \{0\},W \times \{0\}]_{\widetilde{\mathfrak{g}}}=\{0\} \times \mathbb{R}.
\]

As $\widetilde{\mathfrak{g}} \cong W\times \mathbb{R}$ and $\widetilde{\mathfrak{g}}_{CM} \cong H\times \mathbb{R}$, we see that $\widetilde{\mathfrak{g}}$ is a Banach space with the norm
\[ |(w,c) |_{\widetilde{{\mathfrak{g}}}} := \sqrt{\|w\|_W^2 + \vert c\vert^2}. \]
Analogously, we see that $\widetilde{\mathfrak{g}}_{CM}$ is both a Banach space with the norm
\[ |(A,a)|_{\widetilde{\mathfrak{g}}_{CM}} := \sqrt{\|A\|_H^2 + \vert a\vert^2} \]
and a Hilbert space with the inner product
\[
\langle (A,a),(A^{\prime},a^{\prime})\rangle_{\mathfrak{g}_{CM}} :=\langle A,A^{\prime}\rangle_H+aa^{\prime}.
\]

\begin{remark}
In the Banach space topology induced by $| \cdot |_{\widetilde{{\mathfrak{g}}}}$, the Lie algebra $\widetilde{\mathfrak{g}}_{CM}$ is dense in the Lie algebra $\widetilde{\mathfrak{g}}$. Furthermore, in the canonical manifold topology (which is equivalent to the product topology on $\widetilde{G} \cong W \times \mathbb{T}^1$), the Cameron-Martin subgroup $\widetilde{G}_{CM}$ is dense in $\widetilde{G}$.
\end{remark}
Now we can introduce the associated natural quotient maps.

\begin{notation}
We denote  by $\varphi:G \rightarrow \widetilde{G}$ and $\varphi_{CM}:G_{CM} \rightarrow \widetilde{G}_{CM}$ the \emph{quotient maps}.
\end{notation}

Both quotient maps $\varphi$ and $\varphi_{CM}$ are also group homomorphisms.
Under  identification $\widetilde{G}\cong W \times \mathbb{T}^1$ and $\widetilde{G}_{CM} \cong H \times \mathbb{T}^1$, explicit formulae for $\varphi$ and $\varphi_{CM}$ are given by 
\begin{align} \label{eqn.QuotientMap}
\varphi(w,c)=(w,e^{ci})
\end{align}
for any $(w,c) \in G$ and
\begin{align} \label{eqn.QuotientMap.CM}
\varphi_{CM}(w,c)=(w,e^{ci})
\end{align}
for any $(w,c) \in G_{CM}$.

\begin{remark}
Both quotient maps $\varphi:G \rightarrow \widetilde{G}$ and $\varphi_{CM}:G_{CM} \rightarrow \widetilde{G}_{CM}$ are smooth.
\end{remark}

Now  we study differentials of these quotient maps. Recall that the tangent spaces $T_eG$ and $T_{\widetilde{e}}\widetilde{G}$ identified with the space  $W\times \mathbb{R}$ while the tangent spaces $T_eG$ and $T_{\widetilde{e}}\widetilde{G}_{CM}$ can be identified with the space  $H\times \mathbb{R}$.

At the identity $e$, the differential $d\varphi_e:T_eG \cong W\times \mathbb{R} \longrightarrow T_{\widetilde{e}}\widetilde{G} \cong W\times \mathbb{R}$ of $\varphi$ is indeed the identity map, that is, $d\varphi_e=Id$ on $W\times \mathbb{R}$. Moreover, at the identity $e$, $d\varphi_e|_{H}: H \longrightarrow H$ is an isometry on $H$.

Similarly, at the identity $\widetilde{e}$, the differential $d(\varphi_{CM})_{e}:T_eG_{CM} \cong H\times \mathbb{R} \longrightarrow T_{\widetilde{e}}\widetilde{G}_{CM} \cong H\times \mathbb{R}$ of $\varphi_{CM}$ is the identity map, that is, $d\varphi_e=Id$ on $W\times \mathbb{R}$. Moreover, at the identity $e$, $d(\varphi_{CM})_e|_{H}: H \longrightarrow H$ is an isometry on $H$ as well.

\begin{remark}
In our setting, $G$ and $\widetilde{G}$, as well as $G_{CM}$ and $\widetilde{G}_{CM}$, correspond to the same Lie algebra.  Note that in general we do not have Lie's third theorem in infinite dimensions, therefore we cannot use topological constraints such as being simply connected to have a unique Lie group for a given Lie algebra. 

Recall that the tangent space $T_eG \cong W\times \mathbb{R}$ is a Lie algebra, denoted by $\mathfrak{g}$, corresponding to $G$ equipped with the Lie bracket given by
\begin{equation}
\label{eqn.LieBracket.Heisenberg}
[(X_1,V_1), (X_2,V_2)]_{\mathfrak{g}} := (0, \omega(X_1,X_2))
\end{equation}
for any $(X_j,V_j)\in \mathfrak{g} \cong W\times \mathbb{R}$. Also, the tangent space is $T_eG_{CM}$ a Lie algebra, denoted by $\mathfrak{g}_{CM}$, corresponding to $G_{CM}$ can be identified with the space $H\times \mathbb{R}$ equipped with the same Lie bracket as \eqref{eqn.LieBracket.Heisenberg} restricted to $H\times \mathbb{R}$. As we can see, both $d\varphi_e:\mathfrak{g} \cong T_eG \longrightarrow \widetilde{\mathfrak{g}} \cong T_{\widetilde{e}}\widetilde{G}$ and $d(\varphi_{CM})_{e}:\mathfrak{g}_{CM} \cong T_eG_{CM} \longrightarrow \widetilde{\mathfrak{g}}_{CM}\cong  T_{\widetilde{e}}\widetilde{G}_{CM}$ are Lie algebra isomorphisms, which implies our statement.
\end{remark}
We give an explicit formula for the exponential map on $\widetilde{G}$.

\begin{lemma}
The exponential map $\operatorname{exp}_{\widetilde{G}}: \widetilde{\mathfrak{g}} \longrightarrow \widetilde{G}$ is given by
\begin{equation}
\operatorname{exp}_{\widetilde{G}}(A,a)=(A,e^{ai})
\end{equation}
for any $(A,a) \in \widetilde{\mathfrak{g}}$, which is a local diffeomorphism.
\end{lemma}

\begin{proof}
From \cite[Lemma 3.8]{DriverGordina2008}, the exponential map $\operatorname{exp}_{G}:\widetilde{\mathfrak{g}} \longrightarrow G$ is given by
\begin{equation} \label{eqn.ExponentialMap.Heisenberg}
\operatorname{exp}_{G}(A,a)=(A,a)
\end{equation}
for any $(A,a) \in \mathfrak{g}$, which is a global diffeomorphism. Now, for any $X\in\mathfrak{g}$, we know that $\varphi\left(\operatorname{exp}_G(X)\right)=\operatorname{exp}_{\widetilde{G}}\left(d\varphi_e(X)\right)$, as $\varphi$ is a Lie group homomorphism. Since $d\varphi_e=Id$ on $W\times \mathbb{R}$, so is $(d\varphi_e)^{-1}$. In this way, for any $(A,a) \in \widetilde{\mathfrak{g}}$, we have
\begin{align*}
\operatorname{exp}_{\widetilde{G}}\left(A,a\right) & =\varphi\left(\operatorname{exp}_G((d\varphi_e)^{-1}(A,a))\right)
\\
&
=\varphi\left(\operatorname{exp}_G((A,a)\right)
\\
&
=\varphi(A,a)=(A,e^{ai})
\end{align*}
where the third equality follows from \eqref{eqn.ExponentialMap.Heisenberg}. Such an explicit formula shows that the exponential map is a local diffeomorphism.
\end{proof}

\subsection{Examples}

We list some basic examples for the construction of
these infinite-dimensional Heisenberg groups.

\begin{example} [Finite-dimensional groups]
\label{ex.Heis}
Let $W=H\cong\mathbb{R}^{2n}$. Suppose $\omega$ is a symplectic form on $\mathbb{R}^{2n}$. Then $G=\mathbb{R}^{2n}\times\mathbb{R}$ equipped with the group operation defined by \eqref{eqn.GroupLaw.Heisenberg} is a non-isotropic
Heisenberg group (see \cite{GordinaLuo2022} for more details) and $\widetilde{G}=\mathbb{R}^{2n}\times\mathbb{T}^1$ equipped with a group operation as
defined by \eqref{eqn.GroupLaw.Reduced} is a finite-dimensional reduced Heisenberg group.
\end{example}

\begin{example}[Infinite-dimensional groups of a symplectic vector space]
\label{ex.infHeis}
Let $(K,\langle\cdot,\cdot\rangle)$ be a Hilbert space and $Q$ be a
strictly positive trace class operator on $K$.  For $h, k\in K$, let $\langle
h, k\rangle_Q:= \langle h, Qk \rangle$ and $\|h\|_Q:= \sqrt{\langle
h, h\rangle_Q}$, and let $(K_Q,\langle\cdot, \cdot\rangle_Q)$ denote the Hilbert
space completion of $(K,\|\cdot\|_Q)$.
Then $W=(K_Q)_{\operatorname{Re}}$ and $H=K_{\operatorname{Re}}$ determines an abstract Wiener
space (see, for example,  of \cite[Exercise 17 on p.59]{KuoLNM1975}).  Let
\[
\omega( w,z ) := \operatorname{Im}\langle w, z \rangle_Q,
\]
Then $G=(K_Q)_{\operatorname{Re}}\times\mathbb{R}$ equipped with a group operation as
defined by \eqref{eqn.GroupLaw.Heisenberg} is an infinite-dimensional Heisenberg-like group and $\widetilde{G}=(K_Q)_{\operatorname{Re}}\times\mathbb{T}^1$ equipped with a group operation as
defined by \eqref{eqn.GroupLaw.Reduced} is an infinite-dimensional reduced Heisenberg group.
\end{example}

\section{Properties of infinite-dimensional reduced Heisenberg groups} \label{sec.Properties}

In this section we study geometric and analytic properties of infinite-dimensional reduced Heisenberg groups, which are crucial to the logarithmic Sobolev inequalities we prove later.

\subsection{Finite-dimensional projection groups}\label{s.gpproj}

Given an abstract Wiener space $(W,H,\mu)$,
let $i:H\rightarrow W$ be the inclusion map, and $i^*:W^*\rightarrow H^*$ be
its transpose so that $i^*\ell:=\ell\circ i$ for all $\ell\in W^*$.  We take
\[ H_* := \{h\in H: \langle\cdot,h\rangle_H\in \mathrm{Range}(i^*)\subset H^*\},
\]
which is a dense subspace of $H$.

Suppose that $P:H\rightarrow H$ is a finite rank orthogonal projection
such that $PH\subset H_*$. We may extend $P$ to a (unique) continuous operator
from $W\rightarrow H$ and we still denote the extension by $P$. We refer to\cite[Section 3.4]{DriverGordina2008} for more details of the construction of $P$.

Let $\mathrm{Proj}(W)$ denote the collection of finite rank projections
on $W$ such that
\begin{enumerate}
\item $PW\subset H_*$,
\item $P|_H:H\rightarrow H$ is an orthogonal projection, and
\item $PW$ is sufficiently large to satisfy
H\"ormander's condition (that is, $\{\omega(A,B):A,B\in PW\}=\mathbb{R}$).
\end{enumerate}

\begin{notation}
\label{n.proj}

For each $P\in\mathrm{Proj}(W)$, we take $\Gamma_{P}=\{(\mathbf{0},2k\pi):k\in \mathbb{Z}\} \subseteq G_P$. We define $G_P:=
PW\times\mathbb{R}\subset H_*\times\mathbb{R}$ and $\widetilde{G}_P:=G_P\backslash \Gamma_P$. 
\end{notation}

Note that for each $P\in\mathrm{Proj}(W)$, $G_P$ is a subgroup of $G$. By the second isomorphism theorem, $\widetilde{G}_P=G_P\backslash \Gamma_P=G_P\backslash (\Gamma \cap G_P)$ is a subgroup of $\widetilde{G}=G \backslash \Gamma$. Here $\widetilde{G}_P$ can be identified as $\widetilde{G}_P \cong PW\times \mathbb{T}^1$. 

\begin{notation}
For each $P\in\mathrm{Proj}(W)$, we let $\mathfrak{g}_P=\mathrm{Lie}(G_P) = PW\times\mathbb{R}$ and $\widetilde{\mathfrak{g}}_P=\mathrm{Lie}(\widetilde{G}_P) = PW\times\mathbb{R}$.
\end{notation}

As we can see, when restricted to the finite-dimensional setting, $G_P$ and $\widetilde{G}_P$ have the same Lie algebra.

\begin{notation}
For each $P\in\mathrm{Proj}(W)$, we denote the corresponding
projections by $\pi_P:G\rightarrow G_P$ \[ \pi_P(w,x):= (Pw,x). \] 
for any for any $w\in W$ and any $x\in \mathbb{R}^1$, and $\widetilde{\pi}_P:\widetilde{G}\rightarrow \widetilde{G}_P$ \[ \widetilde{\pi}_P(w,s):= (Pw,s) \] 
for any $w\in W$ and any $s\in \mathbb{T}^1$.
\end{notation}

From \cite[p.22-23]{GordinaLuo2022}, $G_P$ is (isomorphic to) a non-isotropic Heisenberg group. Then we see that $\widetilde{G}_P$ is a finite-dimensional reduced Heisenberg group and we also have the finite-dimensional quotient map
\begin{align*}
\varphi_P: & G_P \longrightarrow \widetilde{G}_P
\\
& 
(\mathbf{v},c) \mapsto (\mathbf{v},e^{2\pi i c})
\end{align*}
for any $\mathbf{v}\in PW$ and any $c\in \mathbb{R}$. Direct computation shows that the following diagram
\begin{equation}
\begin{tikzcd} \label{ProjectionDiagram}
G \arrow{r}{\pi_P} \arrow[swap]{d}{\varphi} & G_P \arrow{d}{\varphi_P} \\%
\widetilde{G} \arrow{r}{\widetilde{\pi}_P}& \widetilde{G}_P
\end{tikzcd}
\end{equation}
commutes.

\subsection{Left-invariant vector fields}

We discuss the connection between left-invariant vector fields on $G$ and $\widetilde{G}$.

For any $g\in G$ we denote by $L_{g}:G\rightarrow G$ the left multiplication by $g$. For any $\widetilde{g}\in \widetilde{G}$ we denote by $L_{\widetilde{g}}:\widetilde{G}\rightarrow \widetilde{G}$ the left multiplication by $\widetilde{g}$. 

\begin{notation} 
We denote by $\widetilde{X}$ the left-invariant vector field $\widetilde{X}(g):=dL_{g}X$ induced by $X \in \mathfrak{g}$ on $G$. That is,  for any $X\in \mathfrak{g}$ and any $g \in G$, we have 
\begin{align} \label{eqn.VectorField.Heisenberg}
(\widetilde{X}f)(g):=\left.\frac{d}{dt}\right\vert_{t=0}f\left(g\operatorname{exp}_{G}\left(tX\right)\right).
\end{align}
\end{notation}

\begin{notation} 
We denote by $\widetilde{X}$ the left-invariant vector field $\widetilde{X}(\widetilde{g}):=dL_{\widetilde{g}}X$ induced by $X \in\widetilde{\mathfrak{g}}$ on $\widetilde{G}$. That is,  for any $X\in\widetilde{\mathfrak{g}}$ and any $\widetilde{g} \in \widetilde{G}$, we have 
\begin{align} \label{eqn.VectorFieldonHomogeneousSpace}
(\widetilde{X}f)(\widetilde{g}):=\left.\frac{d}{dt}\right\vert_{t=0}f\left(\widetilde{g}\operatorname{exp}_{\widetilde{G}}\left(tX\right)\right).
\end{align}
Note that we abuse the notation here for the notation of left-invariant vector fields on both $G$ and $\widetilde{G}$.
\end{notation}

\begin{lemma}
For any $X\in W \times \mathbb{R} \cong \mathfrak{g} \cong \widetilde{\mathfrak{g}}$ and any smooth function $f \in C^{\infty}(\widetilde{G})$, we have that for any $\widetilde{g} \in \widetilde{G}$
\begin{align} \label{eqn.VectorRelation1}
\widetilde{X} \left(\widetilde{g}\right)=\left(d\varphi_g(\widetilde{X})\right)(\widetilde{g})
\end{align}
with $\varphi(g)=\widetilde{g}$ for some $g\in G$. In this way, For any $X\in W \times \mathbb{R} \cong \mathfrak{g} \cong \widetilde{\mathfrak{g}}$ and any smooth function $f \in C^{\infty}(\widetilde{G})$
\begin{align}\label{eqn.VectorRelation2}
(\widetilde{X}f)\circ \varphi =\widetilde{X} (f\circ \varphi)
\end{align}
\end{lemma}

\begin{proof}
For any $X\in \widetilde{\mathfrak{g}}$ and any $\widetilde{g},\widetilde{g}^{\prime}\in \widetilde{G}$ where $\widetilde{g}=\varphi(g)$ and $\widetilde{g}^{\prime}=\varphi(g^{\prime})$ for some $g,g^{\prime}\in G$, we have $L_{\widetilde{g}}\widetilde{g}^{\prime}=L_{\widetilde{g}}\varphi(g^{\prime})=\varphi(g)\varphi(g^{\prime})=\varphi(gg^{\prime})=\varphi(L_gg^{\prime})$, as $\varphi$ is a group homomorphism. This means $L_{\widetilde{g}} \circ \varphi=\varphi \circ L_g$ for any $\widetilde{g}=\varphi(g)$. Differentiating at the identity $e$, we have $dL_{\widetilde{g}} \circ d\varphi_e=d\varphi_g \circ dL_g$. For any $X \in \widetilde{\mathfrak{g}}$, we have
\begin{align*}
 \widetilde{X} \left(\widetilde{g}\right) & =\left(dL_{\widetilde{g}} (X)\right)(\widetilde{g})
\\
&
=\left(\left( d\varphi_g \circ dL_g \circ (d\varphi_e)^{-1}\right) (X)\right)(\widetilde{g})
\\
&
=\left(\left( d\varphi_g \circ dL_g\right) (X)\right)(\widetilde{g})
\\
&
=\left(d\varphi_g (\widetilde{X})\right)(\widetilde{g}).
\end{align*}
The third equality is by that $(d\varphi_e)^{-1}$ is an identity map on $W\times \mathbb{R}$. Therefore, \eqref{eqn.VectorRelation1} and \eqref{eqn.VectorRelation2} hold. 
\end{proof}

\subsection{Distance on the Cameron-Martin subgroup}
\label{s.length}

The sub-Riemannian distance on $\widetilde{G}_{CM}$ can be defined similarly to how it is done on $G_{CM}$.

\begin{notation}[Horizontal distances on $G_{CM}$ and $\widetilde{G}_{CM}$]
\label{n.length} Below we define several distances on Cameron-Martin subgroups.

\begin{enumerate}
\item For any $x=(A,a)\in \mathfrak{g}_{CM} \cong \widetilde{\mathfrak{g}}_{CM}$, let
\[
|x|_{\mathfrak{g}_{CM}}^2 : = \|A\|_H^2 + \vert a\vert^2
\]
and 
\[
|x|_{\widetilde{\mathfrak{g}}_{CM}}^2 : = \|A\|_H^2 + \vert a\vert^2.
\]
The \emph{length} of a $C^1$-path $\sigma:[a,b]\rightarrow
G_{CM}$ is defined as
\[
\ell(\sigma)
	:= \int_a^b |dL_{(\sigma(s))^{-1}}\dot{\sigma}(s)|_{\mathfrak{g}_{CM}} \,ds.
\]
The \emph{length} of a $C^1$-path $\gamma:[a,b]\rightarrow
\widetilde{G}_{CM}$ is defined as
\[
\widetilde{\ell}(\gamma)
	:= \int_a^b |dL_{(\gamma(s))^{-1}}\dot{\gamma}(s)|_{\widetilde{\mathfrak{g}}_{CM}} \,ds.
\]

\item \label{i.2}
A $C^1$-path $\sigma:[a,b]\rightarrow G_{CM}$ is  \emph{horizontal} if
$L_{\left(\sigma(t)\right)^{-1}}\dot{\sigma}(t)\in H\times\{0\}$
for a.e.~$t$. A $C^1$-path $\gamma:[a,b]\rightarrow \widetilde{G}_{CM}$ is {\em horizontal} if
$dL_{(\gamma(s))^{-1}}\dot{\gamma}(t)\in H\times\{0\}$
for a.e.~$t$.   Let $C^{1,h}_{CM}$ and $\widetilde{C}^{1,h}_{CM}$ denote the set of horizontal paths
$G_{CM}$ and $\widetilde{G}_{CM}$ respectively.

\item The {\em horizontal distance} between $x,y\in G_{CM}$ is defined by
\[ d(x,y) := \inf\{\ell(\sigma): \sigma\in C^{1,h}_{CM} \text{ such
    that } \sigma(0)=x \text{ and } \sigma(1)=y \}. \]
The {\em horizontal distance} between $\widetilde{x},\widetilde{y}\in \widetilde{G}_{CM}$ is defined by
\[ \widetilde{d}(\widetilde{x},\widetilde{y}) := \inf\{\widetilde{\ell}(\gamma): \sigma\in \widetilde{C}^{1,h}_{CM} \text{ such
    that } \gamma(0)=\widetilde{x} \text{ and } \gamma(1)=\widetilde{y} \}. \]
\end{enumerate}
\end{notation}

\begin{remark}
The topology induced by this distance $d$ is equivalent to the topology induced by a homogeneous norm on $G$ (see \cite[Proposition 2.17, Proposition 2.18]{GordinaMelcher2013} for details).
\end{remark}

Now we explore the connection between $d$ and $\widetilde{d}$.

\begin{proposition}
For any $\widetilde{g},\widetilde{g}^{\prime} \in \widetilde{G}_{CM}$,
\begin{align}
\widetilde{d}(\widetilde{g},\widetilde{g}^{\prime}) \leqslant d(g,g^{\prime})
\end{align}
where $\widetilde{g}=\varphi(g)$ and $\widetilde{g}^{\prime}=\varphi(g^{\prime})$ for some $g,g^{\prime} \in G_{CM}$.
\end{proposition}

\begin{proof}
Given any $\widetilde{g},\widetilde{g}^{\prime} \in \widetilde{G}_{CM}$, we take arbitrary $g,g^{\prime} \in G_{CM}$ with $\widetilde{g}=\varphi(g)$ and $\widetilde{g}^{\prime}=\varphi(g^{\prime})$. For any $\sigma \in C^{1,h}_{CM}$ with $\sigma(0)=g$ and $\sigma(1)=g^{\prime}$, we have $\varphi \circ \sigma \in \widetilde{C}^{1,h}_{CM}$ with $(\varphi \circ \sigma)(0)=\widetilde{g}$ and $(\varphi \circ \sigma)(1)=\widetilde{g}^{\prime}$. Also,
\begin{align*}
\vert dL_{((\varphi \circ \sigma)(s))^{-1}} \dot{(\varphi \circ \sigma)}(s)\vert_{\widetilde{\mathfrak{g}}_{CM}} & = \vert \left(d\varphi_e \circ dL_{(\sigma(t))^{-1}} \right) \dot{\sigma}(s) \vert_{\widetilde{\mathfrak{g}}_{CM}}
\\
&
=\vert dL_{(\sigma(t))^{-1}} \dot{\sigma}(s) \vert_{\mathfrak{g}_{CM}}.
\end{align*}
The second equality holds as $dL_{(\sigma(t))^{-1}} \dot{\sigma}(s) \in H$ and $d\varphi_e|_{H}$ is an isometry on $H$. In this way, we have $\widetilde{l}(\varphi \circ \sigma) \leqslant l(\sigma)$ and thus,
\begin{align*}
\widetilde{d}(\widetilde{g},\widetilde{g}^{\prime}) & \leqslant \inf\{\widetilde{l}(\varphi \circ \sigma):\sigma\in C^{1,h}_{CM} \text{ such
    that } \sigma(0)=g \text{ and } \sigma(1)=g^{\prime}\} 
    \\
    &
    \leqslant \inf\{l(\sigma):\sigma\in C^{1,h}_{CM} \text{ such
    that } \sigma(0)=g \text{ and } \sigma(1)=g^{\prime} \}
    \\
    &
    =d(g,g^{\prime}).
\end{align*}
\end{proof}

\subsection{Sub-Laplacian and horizontal gradient}

We start by defining the sub-Laplacian and the horizontal gradient on $\widetilde{G}$, and then describe their relation to those on $G$.

\begin{definition}
\label{d.cyl}
A function $f:G\rightarrow\mathbb{C}$ is a \emph{(smooth) cylinder function}
if there exist  some $P\in\mathrm{Proj}(W)$
and some (smooth) function $F:G_P\rightarrow\mathbb{R}$ such that $f=F\circ\pi_P$.  We denote by $\mathcal{C}$ the space of smooth cylinder functions on $G$. A \emph{cylinder polynomial} is a cylinder function, $f=F\circ \pi_{P}:G \rightarrow \mathbb{C}$, where $P\in\operatorname*{Proj}(W)$ and $F$ is a real polynomial function on $G_P$. 
\end{definition}

\begin{definition}
We say that $f:\widetilde{G} \rightarrow \mathbb{R}$ is a \emph{(smooth) cylinder function} if there exists some $P\in\mathrm{Proj}(W)$ and some (smooth) function $F:\widetilde{G}_p\rightarrow \mathbb{R}$ such that $f=F\circ \widetilde{\pi}_P$. We say that $f\in \mathcal{FC}^{\infty}_c(\widetilde{G})$ if there exists some $P\in\mathrm{Proj}(W)$ and some $F\in C^{\infty}_c(\widetilde{G}_P)$  such that $f=F\circ \widetilde{\pi}_P$.
\end{definition}

\begin{definition}
Let $\left\{  e_{j}\right\}_{j=1}^{\infty}$ be an orthonormal
basis for $H$. Given any smooth cylinder function $f:G\rightarrow
\mathbb{R}$, we define the \emph{sub-Laplacian} $L_H^G$ by
\begin{align} \label{SubellipticLaplacian}
L_H^Gf(g)  :=\sum_{j=1}^{\infty}\left[  \widetilde{\left(
e_{j},0\right)  }^{2}f\right]  (g).
\end{align}
for any $g\in G$. Given any smooth cylinder function $f:\widetilde{G}\rightarrow
\mathbb{R}$, we define the \emph{sub-Laplacian} $L_H^{\widetilde{G}}$ on $\widetilde{G}$ is
\begin{align*}
(L_H^{\widetilde{G}}f)(\widetilde{g})=\sum_{j=1}^{\infty} \left[\widetilde{\left(
e_{j},0\right) }^{2} f\right](\widetilde{g})
\end{align*}
for any $\widetilde{g} \in \widetilde{G}$.
\end{definition}

\begin{notation}
We denote
\begin{align*}
& \mathcal{A}:=\left\{f\in \mathcal{C}: f \text{ and }  Xf  \text{ are polynomially bounded } \right.
\\
& \left. \text{ for any left-invariant vector field }X  \text{ on }  G
\right\}.
\end{align*}
\end{notation}

\begin{definition}
For any $u \in \mathcal{A}$, define the \emph{horizontal gradient} $\operatorname{grad}^G_Hu:G \rightarrow H$ of $u$ by
\begin{align} \label{eqn.HorizontalGradientInfinite}
\langle \operatorname{grad}^G_Hu,h\rangle_H=\widetilde{\left(h,0\right)}u
\end{align}
for any $h\in H$. For any $f\in \mathcal{FC}^{\infty}_c(\widetilde{G})$, the \emph{horizontal gradient} $\operatorname{grad}^{\widetilde{G}}_Hf:\widetilde{G} \rightarrow H$ of $f$ on $\widetilde{G}$ is
\begin{align} \label{eqn.HorizontalGradientInfiniteReduced}
\langle \operatorname{grad}^{\widetilde{G}}_Hf,h\rangle_H=\widetilde{\left(h,0\right)}f
\end{align}
for any $h\in H$.

\end{definition}

Based on the definition of the sub-Laplacian and the horizontal gradient on $\widetilde{G}$, we obtain the following proposition by direct computation.

\begin{proposition}
For any $f\in \mathcal{FC}^{\infty}_c(\widetilde{G})$, we have  
\begin{align} \label{eqn.GradientRelation}
\langle \operatorname{grad}^{\widetilde{G}}_Hf,\operatorname{grad}^{\widetilde{G}}_Hf\rangle_H \circ \varphi=\langle \operatorname{grad}^G_H (f\circ \varphi),\operatorname{grad}^G_H(f \circ \varphi) \rangle_H
\end{align}
and
\begin{align} \label{eqn.LaplacianRelation}
\left(L_{H}^{\widetilde{G}} f\right) \circ \varphi=L^{G}_{H}\left(f \circ \varphi\right).
\end{align}
\end{proposition}

\begin{proof}
Let $\left\{  e_{j}\right\}_{j=1}^{\infty}$ be an orthonormal
basis for $H$. By the definition of the horizontal gradient on $G$ and $\widetilde{G}$, we have
\begin{align*}
\langle \operatorname{grad}^G_Hu, \operatorname{grad}^G_Hu\rangle =\sum_{j=1}^{\infty}\left(  \widetilde{\left(
		e_{j},0\right)^2  }u\right) .
\end{align*}
for any $u\in \mathcal{A}$ and
\begin{align*}
\langle \operatorname{grad}^{\widetilde{G}}_Hf,\operatorname{grad}^{\widetilde{G}}_Hf\rangle_H=\sum_{j=1}^{\infty} \left(\widetilde{\left(
e_{j},0\right) }f\right)^2
\end{align*}
for any $f\in \mathcal{FC}^{\infty}_c(\widetilde{G})$. In addition, for any $f\in \mathcal{FC}^{\infty}_c(\widetilde{G})$ and any left-invariant vector field $X$ on $G$, we have $(X(f\circ \varphi))(g)=\left(((d\varphi_gX)f)\circ \varphi
\right)(g)$ for any $g\in G$ by \eqref{eqn.VectorRelation1}, so $f\circ \varphi \in \mathcal{A}$. In this way, by \eqref{eqn.VectorRelation1} and \eqref{eqn.VectorRelation2}, we see that
\begin{align*} 
& \left( \langle \operatorname{grad}^{\widetilde{G}}_Hf,\operatorname{grad}^{\widetilde{G}}_Hf\rangle_H \circ \varphi\right)(g) =\left( \left(\sum_{j=1}^{\infty} \left(\widetilde{\left(
e_{j},0\right) }f\right)^2\right) \circ \varphi\right)(g)
\\
&
=\sum_{j=1}^{\infty}\left(  d\varphi_g\left(\widetilde{\left(
		e_{j},0\right)}\right)f\right)^2 \left(  \varphi(g)\right)
=\sum_{j=1}^{\infty} \left(\widetilde{\left(
		e_{j},0\right)} (f\circ \varphi)\right)^2(g)
\\
&
=\left(\langle \operatorname{grad}^G_H (f\circ \varphi),\operatorname{grad}^G_H(f \circ \varphi) \rangle_H\right)(g)
\end{align*}
and
\begin{align*}
\label{df.SubLaplacianonQuotientSpace}
& \left( (L_H^{\widetilde{G}}f) \circ \varphi\right)(g) =\left( \left(\sum_{j=1}^{\infty} \widetilde{\left(
e_{j},0\right) }^{2} f \right) \circ \varphi\right) (g)
\\
&
=\sum_{j=1}^{\infty}\left[\left(d\varphi_g  \left(\widetilde{\left(
e_{j},0\right) }\right)\right)^{2}f\right]  (\varphi(g))
=\sum_{j=1}^{\infty} \left(\widetilde{\left(
		e_{j},0\right)}^2 (f\circ \varphi)\right)(g)
\\
&
=\left(L_{H}^G f \circ \varphi\right)(g)
\end{align*}
for $f\in \mathcal{FC}^{\infty}_c(\widetilde{G})$ and any $g\in G$.
\end{proof}

\begin{remark}
By \cite[Proposition 3.17]{BaudoinGordinaMelcher2013}, the sub-Laplacian in \eqref{SubellipticLaplacian} is well-defined since it only depends on $\langle \cdot,\cdot\rangle_H$ and it is independent of the choice of the orthonormal basis. By \eqref{eqn.LaplacianRelation}, we see that the sub-Laplacian $L_{H}^{\widetilde{G}}$ is a well-defined operator on $\widetilde{G}$ as well, which only depends on $\langle \cdot,\cdot\rangle_H$ and is independent of the choice of the orthonormal basis.
\end{remark}

\subsection{Hypoelliptic Brownian motion and hypoelliptic heat kernel measure}\label{sec.HeatKernelMeasureInfinite}

We now define the hypoelliptic Brownian motion and the hypoelliptic heat kernel measure on $\widetilde{G}$. Let $\{B_t\}_{t\ge0}$ be a Brownian motion on $W$ with variance determined by
\[
\mathbb{E}\left[\langle B_s,h\rangle_H \langle B_t,k\rangle_H\right]
    = \langle h,k \rangle_H \min(s,t),
\]
for all $s,t\geqslant 0$ and $h,k\in H_*$. 

\begin{definition}
A \emph{hypoelliptic Brownian motion} on $G$ is the continuous $G$-valued process given by
\[g_t^G = \left( B_t, \frac{1}{2}\int_0^t \omega(B_s,dB_s)\right).
\]
A \emph{hypoelliptic Brownian motion} on $\widetilde{G}$ is the continuous $\widetilde{G}$-valued process given by
\[g_{t}^{\widetilde{G}}= \left( B_t, e^{\frac{i}{2}\int_0^t \omega(B_s,dB_s)}\right).
\]
\end{definition}

Let $g_t^G$ be the horizontal Brownian motion on $G$. We have
\begin{align} \label{eqn.BM.Reduced}
g_{t}^{\widetilde{G}}=\varphi\left(g_t^G\right).
\end{align}

\begin{definition}
\label{df.HeatKernelMeasureInfinite}
We call a family of measures $\{\mu^G_t\}_{t>0}$ on $G$ defined by $\mu^G_t=\mathrm{Law}(g^G_{t})$ for any $t>0$ the {\em hypoelliptic heat kernel measure
} on $G$ at time $t$. We call a family of measures $\{\mu_t^{\widetilde{G}}\}_{t>0}$ on $\widetilde{G}$ defined by $\mu_t^{\widetilde{G}}=\mathrm{Law}(g_{t}^{\widetilde{G}})$ for any $t>0$ the \emph{hypoelliptic heat kernel measure} on $\widetilde{G}$
at time $t$. 
\end{definition}

This definition and \eqref{eqn.BM.Reduced} directly imply the following characterization of the hypoelliptic heat kernel measure on $\widetilde{G}$.

\begin{theorem} \label{thm.HeatKernelMeasureCharacterization}
The measure $\{\mu_t^{\widetilde{G}}\}_{t>0}$ is the pushforward measure $\{\varphi_{\#}\mu_t^G\}_{t>0}$ to $\widetilde{G}$ by $\varphi$. 
\end{theorem}
In this way, we can obtain some properties of $\mu_t^{\widetilde{G}}$ inherited from $G$.

\begin{proposition}
For all $t>0$, $\mu_t^{\widetilde{G}}(\widetilde{G}_{CM})=0$. 
\end{proposition}

\begin{proof}
By \cite[Proposition 2.30]{GordinaMelcher2013} for all $t>0$ we have $\mu^G_t(G_{CM})=0$. By Theorem~\ref{thm.HeatKernelMeasureCharacterization}, we have
\begin{align*}
\mu_t^{\widetilde{G}}(\widetilde{G}_{CM})=\mu^G_t(G_{CM})=0.
\end{align*}
\end{proof}

\begin{theorem} 
For any $f=F\circ \widetilde{\pi}_P$ with some $P\in\mathrm{Proj}(W)$ and some $F\in C^2(\widetilde{G}_P)$ and any $t>0$, the following heat equation holds
\begin{align*}
& \frac{d}{dt} \int_{\widetilde{G}} fd\mu_t^{\widetilde{G}}=\frac{1}{2}\int_{\widetilde{G}} L_{\mathcal{H}}^{\widetilde{G}}fd\mu_t^{\widetilde{G}},
\\
&
\lim_{t \to 0} \int_{\widetilde{G}} fd\mu_t^{\widetilde{G}}=f(\widetilde{e}).
\end{align*}
\end{theorem}

\begin{proof}
For any $f=F\circ \widetilde{\pi}_P$ with some $P\in\mathrm{Proj}(W)$ and some $F\in C^2(\widetilde{G}_P)$, we have $f\circ \varphi=F\circ \widetilde{\pi}_P\circ \varphi=F\circ \varphi_P \circ \pi_P$ by the commutativity of the diagram \eqref{ProjectionDiagram}, so we see that $F\circ \varphi_P \in C^2(G_P)$. From \cite[Corollary 5.7]{BaudoinGordinaMelcher2013}, We know that $\mu^G_t$ satisfies the heat equation on $G$, that is, the following heat equation holds for such $f\circ \varphi=(F\circ \varphi_P) \circ \pi_P$ with $F\circ \varphi_P \in C^2(G_P)$
\begin{align*}
& \frac{d}{dt} \int_{G} (f\circ \varphi) d\mu_t^{G}=\frac{1}{2}\int_{G} L_{\mathcal{H}}^{G}(f\circ \varphi) d\mu_t^{G},
\\
&
\lim_{t \to 0} \int_{G} (f\circ \varphi)d\mu_t^{G}=(f\circ \varphi)(e).
\end{align*}
Using the change of variable formula and Theorem~\ref{thm.HeatKernelMeasureCharacterization}, we can see that $\mu^{\widetilde{G}}_t$ satisfies the above given heat equation on $\widetilde{G}$.
\end{proof}

Moreover, by the change of variable formula, we have
\begin{align} \label{eqn.L^2NormRelation}
\Vert f\Vert_{L^2\left(\widetilde{G}, d\mu_t^{\widetilde{G}}\right)}=\Vert f\circ \varphi\Vert_{L^2\left(G, d\mu_t^{G}\right)}
\end{align}
for any $f\in L^2\left(\widetilde{G}, d\mu_t^{\widetilde{G}}\right)$.

\subsection{Dirichlet forms}
\subsubsection{Dirichlet form on $G$}
Recall some basics on  Dirichlet forms on $G$.

\begin{definition}
For any $u,v\in \mathcal{A}$, we define a \emph{pre-Dirichlet form} by
\begin{align} \label{eqn.PreDirichletForm.Heisenberg}
\mathcal{E}^0_{G}(u,v):=\int_G \langle \operatorname{grad}^G_Hu,\operatorname{grad}^G_Hv \rangle_H d\mu^G_t.  
\end{align}
Then the \emph{Dirichlet norm} $\Vert \cdot \Vert_{\mathcal{E}_{G}}$ on $\mathcal{A}$ is given by
\begin{align}
\Vert f \Vert_{\mathcal{E}_{G}}^2:=\Vert f\Vert^2_{L^2\left(G,\mu_t^G\right)}+\mathcal{E}^0_{G}(f,f)
\end{align}
for any $f\in \mathcal{A}$.
\end{definition}

\begin{remark}
When we restrict $\mathcal{E}^0_{G}$ to the collection of cylinder polynomials, our definition of $\mathcal{E}^0_{G}$ coincides with the one in \cite[(6.11)]{GordinaLuo2022}. 
\end{remark}

Similarly to the proof of \cite[Theorem 6.22]{GordinaLuo2022}, we obtain the following theorem.

\begin{theorem} \label{thm.DirichletHeisenberg}
The pre-Dirichlet form $\mathcal{E}^{0}_G$ is closable and its closure, $\mathcal{E}_G$, is a Dirichlet form on $L^2\left(G, d\mu^G_t\right)$.
\end{theorem}

\begin{proof}
Similarly to the proof of \cite[Theorem 6.22]{GordinaLuo2022}, the closability of $\mathcal{E}^{0}_G$ can be obtained. Here we only need to show that closure of $\mathcal{E}^{0}_G$, $\mathcal{E}_G$, is Markovian on $L^2\left(G,d\mu^G_t\right)$. By \cite[Theorem 3.1.1]{FukushimaOshimaTakedaBook2011}, it suffices to show $\mathcal{E}_G$ is Markovian on $\mathcal{FC}^{\infty}_c(G)$, as $\mathcal{E}_G$ is well-defined on $\mathcal{FC}^{\infty}_c(G)$ and $\mathcal{FC}^{\infty}_c(G)$ is dense in $L^2\left(G,d\mu^G_t\right)$. For each $\varepsilon>0$, by  \cite[Exercise 1.2.1]{FukushimaOshimaTakedaBook2011} we can find an infinitely differentiable function $\phi_{\varepsilon}(t)$ such that $\phi_{\varepsilon}(t)=t$ for $t\in [0,1]$, $-\varepsilon\leqslant \phi_{\varepsilon}(t) \leqslant 1+\varepsilon$ for any $t\in \mathbb{R}$ and $0\leqslant \phi_{\varepsilon}(t^{\prime})-\phi_{\varepsilon}(t)\leqslant t^{\prime}-t$ whenever $t<t^{\prime}$. By the chain rule for the Fr\'echet derivative, we have $\langle \operatorname{grad}^G_H (\phi_{\varepsilon}(u)),\operatorname{grad}^G_H(\phi_{\varepsilon}(u)) \rangle_H =\vert \phi_{\varepsilon}^{\prime} (u)\vert^2\langle \operatorname{grad}^G_Hu,\operatorname{grad}^G_Hu \rangle_H $ for any $u\in \mathcal{A}$ and thus
\begin{align*}
\mathcal{E}_{G}(\phi_{\varepsilon}(u),\phi_{\varepsilon}(u)) & =\int_G \vert \phi_{\varepsilon}^{\prime} (u)\vert^2\langle \operatorname{grad}^G_Hu,\operatorname{grad}^G_Hu \rangle_Hd\mu^G_t
\\
&
\leqslant \int_G \langle \operatorname{grad}^G_Hu,\operatorname{grad}^G_Hu \rangle_Hd\mu^G_t,
\end{align*}
as $0 \leqslant \vert \phi^{\prime}_{\varepsilon} \vert \leqslant 1$ (see \cite[Exercise 1.1.2]{FukushimaOshimaTakedaBook2011}). Therefore, $\mathcal{E}_G$, is a Dirichlet form on $L^2\left(G,d\mu^G_t\right)$.
\end{proof}

\begin{notation}
We denote by $\mathcal{D}(\mathcal{E}_G)$ the domain of $\mathcal{E}_G$.
\end{notation}

\begin{remark}
As the collection of cylinder polynomials is a subset of $\mathcal{A}$ and it is also dense in $\mathcal{D}(\mathcal{E}_G) \subseteq L^2\left(G,d\mu^G_t\right)$, we see that the closure of $\mathcal{E}_G^0$ that we obtained from \cite[Theorem 6.22]{GordinaLuo2022} is the same as $\mathcal{E}_G$ here.
\end{remark}

\subsubsection{Dirichlet form on $\widetilde{G}$}

We construct a Dirichlet form with respect to the heat kernel measure $\mu_t^{\widetilde{G}}$ on $\widetilde{G}$ and study its connection to $\mathcal{E}_G$ on $G$.

\begin{definition}
We define a pre-Dirichlet form on $\widetilde{G}$ by
\begin{align}
\mathcal{E}^0_{\widetilde{G}}(f_1, f_2):=\int_{\widetilde{G}} \left\langle \operatorname{grad}^{\widetilde{G}}_H f_1,\operatorname{grad}^{\widetilde{G}}_H f_2\right\rangle_{\mathcal{H}} d\mu_t^{\widetilde{G}}
\end{align}
for any $f_1, f_2\in \mathcal{FC}^{\infty}_c(\widetilde{G})$.
\end{definition}

\begin{notation}
For any $f \in \mathcal{FC}^{\infty}_c(\widetilde{G})$, we denote $\mathcal{E}^0_{\widetilde{G}}(f,f)$ by $\mathcal{E}^0_{\widetilde{G}}(f)$.
\end{notation}

\begin{lemma} \label{lem.DomainRelation}
For any $f_1, f_2 \in \mathcal{FC}^{\infty}_c(\widetilde{G})$, we have $f_1\circ\varphi, f_2\circ\varphi \in \mathcal{D}(\mathcal{E}_{G})$ and
\begin{align} \label{eqn.DirichletFormRelation}
\mathcal{E}^0_{\widetilde{G}}\left(f_1, f_2\right)=\mathcal{E}_{G}\left( f_1\circ\varphi, f_2\circ\varphi\right).
\end{align}
\end{lemma}

\begin{proof}
For any $f\in \mathcal{FC}^{\infty}_c(\widetilde{G})$ and any left-invariant vector field $X$ on $G$, we have $(X(f\circ \varphi))(g)=\left(((d\varphi_gX)f)\circ \varphi
\right)(g)$ for any $g\in G$ by \eqref{eqn.VectorRelation1}, so $f\circ \varphi \in \mathcal{A} \subseteq \mathcal{D}(\mathcal{E}_{G})$. Then we can use the change of variable formula and \eqref{eqn.GradientRelation} to obtain 
\begin{align*}
\mathcal{E}^0_{\widetilde{G}}\left(f\right)=\mathcal{E}_{G}\left(f\circ\varphi\right),
\end{align*}
which implies \eqref{eqn.DirichletFormRelation} using polarization.
\end{proof}

By \eqref{eqn.L^2NormRelation} and \eqref{eqn.DirichletFormRelation}, we have 

\begin{align} \label{eqn.DirichletNormRelation}
\Vert f\Vert_{\mathcal{E}^0_{\widetilde{G}}}=\Vert f\circ \varphi\Vert_{\mathcal{E}_{G}}
\end{align}
for any $f\in \mathcal{FC}^{\infty}_c(\widetilde{G})$.

\begin{theorem} \label{thm.ClosabilityRelation}
The bilinear form $\mathcal{E}^0_{\widetilde{G}}$ is closable on $L^2\left(\widetilde{G},d\mu_t^{\widetilde{G}}\right)$.
\end{theorem}

\begin{proof}
The closability of $\mathcal{E}^0_{\widetilde{G}}$ comes from the closedness of $\mathcal{E}_{G}$. Let $\{f_k\}_{k=1}^{\infty}$ be a sequence in $\mathcal{FC}^{\infty}_c(\widetilde{G})$ such that $\Vert f_k\Vert_{L^2\left(\widetilde{G}, d\mu_t^{\widetilde{G}}\right)} \xrightarrow[
]{k\to\infty} 0$ and $\mathcal{E}^0_{\widetilde{G}}(f_k-f_l) \xrightarrow[]{k, l \to\infty}  0$. Then we have $f_k\circ \varphi, f_l\circ \varphi \in \mathcal{D}\left(\mathcal{E}_{G}\right)$ by Lemma~\ref{lem.DomainRelation} with
\begin{align*}
\Vert f_k\circ \varphi\Vert_{L^2\left(G, d\mu_t^{G}\right)}=\Vert f_k\Vert_{L^2\left(\widetilde{G}, d\mu_t^{\widetilde{G}}\right)} \xrightarrow[]{k\to\infty} 0
\end{align*}
by \eqref{eqn.L^2NormRelation} and
\begin{align*}
\mathcal{E}_{G}(f_k\circ\varphi-f_l\circ\varphi)=\mathcal{E}^0_{\widetilde{G}}(f_k-f_l) \xrightarrow[]{k, l \to\infty} 0
\end{align*}
by \eqref{eqn.DirichletFormRelation}. Using that $\mathcal{E}_{G}$ is closed, we see that
\begin{align*}
\mathcal{E}^0_{\widetilde{G}}(f_k)=\mathcal{E}_{G}(f_k\circ\varphi) \xrightarrow[]{k\to\infty} 0,
\end{align*}
and therefore $\mathcal{E}^0_{\widetilde{G}}$ is closable on $L^2\left(\widetilde{G}, d\mu_t^{\widetilde{G}}\right)$.   
\end{proof}

\begin{notation}
We denote the closure of $\mathcal{E}^0_{\widetilde{G}}$ on $L^2\left(\widetilde{G},d\mu_t^{\widetilde{G}}\right)$ by $\mathcal{E}_{\widetilde{G}}$ and its domain by $\mathcal{D}\left(\mathcal{E}_{\widetilde{G}}\right) \subseteq L^2\left(\widetilde{G},d\mu_t^{\widetilde{G}}\right)$.
\end{notation}

\begin{proposition}
The closed bilinear form $\mathcal{E}_{\widetilde{G}}$ is a Dirichlet form on $L^2\left(\widetilde{G},d\mu_t^{\widetilde{G}}\right)$.
\end{proposition}

\begin{proof}
It suffices to show that $\mathcal{E}^0_{\widetilde{G}}$ is Markovian. Then, by \cite[Theorem 3.1.1]{FukushimaOshimaTakedaBook2011}, the closability of $\mathcal{E}^0_{\widetilde{G}}$ implies that $\mathcal{E}_{\widetilde{G}}$ is Markovian too. For each $\varepsilon>0$, by  \cite[Exercise 1.2.1]{FukushimaOshimaTakedaBook2011} we can find an infinitely differentiable function $\phi_{\varepsilon}(t)$ such that $\phi_{\varepsilon}(t)=t$ for $t\in [0,1]$, $-\varepsilon\leqslant \phi_{\varepsilon}(t) \leqslant 1+\varepsilon$ for any $t\in \mathbb{R}$ and $0\leqslant \phi_{\varepsilon}(t^{\prime})-\phi_{\varepsilon}(t)\leqslant t^{\prime}-t$ whenever $t<t^{\prime}$. For any $f\in \mathcal{FC}^{\infty}_c(\widetilde{G})$, we have $\phi_{\varepsilon}(f)\in \mathcal{FC}^{\infty}_c(\widetilde{G})$ and $\phi_{\varepsilon}(f\circ \varphi)=\phi_{\varepsilon}(f)\circ \varphi \in \mathcal{A} \subseteq \mathcal{D}\left(\mathcal{E}_{G}\right)$. Then, the Markovian property of $\mathcal{E}_{G}$ together with \eqref{eqn.DirichletFormRelation} implies that
\begin{align*}
\mathcal{E}^0_{\widetilde{G}}\left(\phi_{\varepsilon}(f)\right) & =\mathcal{E}_{G}\left(\phi_{\varepsilon}(f\circ \varphi)\right)
\\
&
\leqslant \mathcal{E}_{G}(f\circ \varphi),
\end{align*}
so $\mathcal{E}^0_{\widetilde{G}}$ is Markovian and thus $\mathcal{E}_{\widetilde{G}}$ is Markovian. Thus $\mathcal{E}_{\widetilde{G}}$ is a Dirichlet form.
\end{proof}

We give the connection between $\mathcal{E}_{\widetilde{G}}$ and $\mathcal{E}_{G}$ in the following theorem. The next theorem shows that the quotient map $\varphi$ is a quasi-homeomorphism between two (quasi-regular) Dirichlet forms.

\begin{theorem}[Dirichlet forms under the quotient map] \label{thm.DirichletFormRelation}
If $f\in \mathcal{D}\left(\mathcal{E}_{\widetilde{G}}\right)$, then $f\circ \varphi \in \mathcal{D}\left(\mathcal{E}_{G}\right)$. In addition, we have
\begin{align} \label{eqn.DirichletFormRelationII}
\mathcal{E}_{\widetilde{G}}\left(f_1,f_2\right)=\mathcal{E}_{G}\left(f_1\circ\varphi, f_2\circ\varphi\right)
\end{align}
for any $f_1, f_2 \in \mathcal{D}\left(\mathcal{E}_{\widetilde{G}}\right)$.
\end{theorem}

\begin{proof}
For any $f\in \mathcal{D}\left(\mathcal{E}_{\widetilde{G}}\right)$, there exists a sequence $\{f_k\}_{k=1}^{\infty} \subseteq \mathcal{FC}^{\infty}_c(\widetilde{G})$ such that $f_k\xrightarrow[]{k\to\infty} 
f$ under $\Vert \cdot\Vert_{\mathcal{E}_{\widetilde{G}}}$. Using \eqref{eqn.DirichletNormRelation} and the fact that $\mathcal{E}_{G}$ is closed, we obtain the convergence of the sequence $\{f_k\circ \varphi \}_{k=1}^{\infty}$ in $\mathcal{D}\left(\mathcal{E}_{G}\right)$ under $\Vert \cdot\Vert_{\mathcal{E}_{G}}$. This means that there exists an $\tilde{f}\in \mathcal{D}\left(\mathcal{E}_{G}\right)$ such that $\Vert f_k\circ \varphi \Vert_{\mathcal{E}_{G}} \xrightarrow[]{k\to\infty} \left\Vert \widetilde{f}\right\Vert_{\mathcal{E}_{G}}$. More precisely, we have
\begin{align*}
\Vert f_k\circ \varphi \Vert_{L^2\left(G, d\mu_t^G\right)} \xrightarrow[]{k\to\infty} \left\Vert \widetilde{f}\right\Vert_{L^2\left(G, d\mu_t^G\right)}, \, \mathcal{E}_{G}(f_k\circ \varphi) \xrightarrow[]{k\to\infty} \mathcal{E}_{G}\left(\tilde{f}\right).
\end{align*}
However, by \eqref{eqn.L^2NormRelation} we have 
\begin{align*}
\Vert f_k\circ \varphi \Vert_{L^2\left(G, d\mu_t^G\right)} & =\Vert f_k\Vert_{L^2\left(\widetilde{G}, d\mu_t^{\widetilde{G}}\right)} 
\\
&
\xrightarrow[]{k\to\infty} \Vert f\Vert_{L^2\left(\widetilde{G}, d\mu_t^{\widetilde{G}}\right)}=\Vert f\circ \varphi \Vert_{L^2\left(G, d\mu_t^G\right)}.
\end{align*}
By the uniqueness of the limit in $L^2\left(G, d\mu_t^G\right)$, we see that $f\circ \varphi=\widetilde{f}$ in $L^2\left(G, d\mu_t^G\right)$. Since $\mathcal{D}\left(\mathcal{E}_{G}\right)$ is a subspace of $L^2\left(G, d\mu_t^G\right)$ and $\widetilde{f}\in \mathcal{D}\left(\mathcal{E}_{G}\right)$, we obtain that $f\circ \varphi=\widetilde{f}$ in $\mathcal{D}\left(\mathcal{E}_{G}\right)$.

Now we prove \eqref{eqn.DirichletFormRelationII}. It suffices to prove it when $f_1=f_2 \in \mathcal{D}\left(\mathcal{E}_{\widetilde{G}}\right)$, from which we can obtain \eqref{eqn.DirichletFormRelationII} using polarization. Using the previous convergence of $f_k \xrightarrow[]{k\to\infty} f$ in $\mathcal{D}\left(\mathcal{E}_{\widetilde{G}}\right)$ under $\Vert \cdot\Vert_{\mathcal{E}_{\widetilde{G}}}$ and $f_k\circ \varphi \xrightarrow[]{k\to\infty} f\circ \varphi$ in $\mathcal{D}\left(\mathcal{E}_{G}\right)$ under $\Vert \cdot\Vert_{\mathcal{E}_{G}}$, we have
\begin{align*}
\mathcal{E}_{G}\left(f\circ\varphi\right)=\mathcal{E}_{G}\left(\tilde{f}\right)=\lim_{k\to\infty}\mathcal{E}_{G}(f_k\circ \varphi)=\lim_{k\to\infty}\mathcal{E}_{\widetilde{G}}(f_k)=\mathcal{E}_{\widetilde{G}}(f)
\end{align*}
implying the result.        
\end{proof}

\section{Logarithmic Sobolev inequalities on infinite-dimensional reduced Heisenberg groups} \label{sec.LSI.Reduced}

In this section, we use the fact that Dirichlet spaces on $G$ and $\widetilde{G}$ are quasi-homeomorphic under the quotient map to prove a hypoelliptic logarithmic Sobolev inequality on $\widetilde{G}$. Moreover, we can control the  LSI constant on the quotient space.

We first recall the hypoelliptic logarithmic Sobolev inequality on $G$.

\begin{theorem} [Theorem 6.24 in \cite{GordinaLuo2022}] \label{thm.LSIInfinite.Reduced} 
The logarithmic Sobolev inequality holds on $G$, that is, for any $f\in \mathcal{D}\left(\mathcal{E}_G\right)$ and any $t>0$
\begin{align} \label{LSIInfinite}
\int_{G}f^2\log f^2d\mu^G_t-\left(\int_{G}f^2 d\mu^G_t\right)\log\left(\int_{G}f^2d\mu^G_t\right)
\leqslant C\left(\omega, t\right) \mathcal{E}_G\left(f,f\right),
\end{align}
where the LSI constant can be chosen to be $C\left(\omega,t\right)=Ct$ with some $C>0$ independent of $\omega$ and thus LSI constant is independent of $\omega$.
\end{theorem}

\begin{remark}
The LSI constant $C\left(\omega,t\right)$ on $G$ can be chosen to be the same as the one on the three-dimensional isotropic Heisenberg group as pointed out in \cite[Theorem 2.8]{GordinaLuo2022} or \cite[Corollaire 1.2]{LiHong-Quan2006}.
\end{remark}

Now we proceed to prove the hypoelliptic logarithmic Sobolev inequality on $\widetilde{G}$. 

\begin{theorem} \label{thm.LSI.Reduced}
The logarithmic Sobolev inequality holds on $\widetilde{G}$, that is, for any $f\in \mathcal{D}\left(\mathcal{E}_{\widetilde{G}}\right)$ and any $t>0$
\begin{align} \label{LSIInfinite.Reduced}
\int_{\widetilde{G}}f^2\log f^2d\mu^{\widetilde{G}}_t-\left(\int_{\widetilde{G}}f^2 d\mu^{\widetilde{G}}_t\right)\log\left(\int_{\widetilde{G}}f^2d\mu^{\widetilde{G}}_t\right)
\leqslant C\left(\widetilde{G}, \mathcal{D}\left(\mathcal{E}_{\widetilde{G}}\right)\right) \mathcal{E}_{\widetilde{G}}\left(f,f\right),
\end{align}
where the LSI constant can be chosen to be 
\begin{align}
C\left(\widetilde{G}, \mathcal{D}\left(\mathcal{E}_{\widetilde{G}}\right)\right) =C\left(\omega,t\right)=Ct
\end{align} 
with some $C>0$ independent of $\omega$ and thus the LSI constant is independent of $\omega$.
\end{theorem}

\begin{proof} Note that this theorem follows from a general statement in Theorem~\ref{thm.LSI.Quasi-homeomorphism}. In addition, we provide a direct proof of this theorem as follows. By \cite[p. 237]{BakryGentilLedouxBook} or by a standard  limiting argument similarly to  \cite[Example 2.7]{Gross1993b}, it suffices to show that the logarithmic Sobolev inequality holds for any $f\in \mathcal{FC}^{\infty}_c(\widetilde{G})$. For any $f\in \mathcal{FC}^{\infty}_c(\widetilde{G})$, Lemma~\ref{lem.DomainRelation} implies that $f\circ \varphi \in \mathcal{D}\left(\mathcal{E}_{G}\right)$. As in Theorem~\ref{thm.LSIInfinite.Reduced} (see \cite[Theorem 6.24]{GordinaLuo2022}), the logarithmic Sobolev inequality holds on $G$. Using the change of variable formula (see \cite[Theorem 3.6.1]{BogachevBook2007}) for the quotient map $\varphi(g)\longmapsto \widetilde{g}\in \widetilde{G}$ in the following logarithmic Sobolev inequality for $f\circ \varphi$
\begin{align*} 
& \int_{G}(f\circ \varphi)^2\log (f\circ \varphi)^2d\mu_t^{G}-\left(\int_{G}(f\circ \varphi)^2 d\mu_t^{G}\right)\log\left(\int_{G}(f\circ \varphi)^2d\mu_t^{G}\right) 
\\
&
\leqslant C\left(\omega,t\right)\mathcal{E}_{G}(f\circ \varphi), 
\end{align*}
we obtain that 
\begin{align} \label{ineq.LSI.Reduced}
\int_{\widetilde{G}}f^2\log f^2d\mu_t^{\widetilde{G}}-\left(\int_{\widetilde{G}}f^2 d\mu_t^{\widetilde{G}}\right)\log\left(\int_{\widetilde{G}}f^2d\mu_t^{\widetilde{G}}\right) \leqslant C\left(\omega,t\right) \mathcal{E}_{\widetilde{G}}(f,f). 
\end{align}
Moreover, we can choose
\begin{align*}
C\left(\widetilde{G}, \mu_t^{\widetilde{G}}\right)=C\left(\omega,t\right)=Ct
\end{align*}
as we can see in \eqref{ineq.LSI.Reduced}, which is independent of $\omega$.   
\end{proof}

\begin{remark}
From the proof of Theorem~\ref{thm.LSI.Reduced}, we observe that the LSI constant $C\left(\widetilde{G}, \mu_t^{\widetilde{G}}\right)$ on $\widetilde{G}$ can be chosen to be the one on the three-dimensional isotropic Heisenberg group in \cite[Theorem 2.8]{GordinaLuo2022} or \cite[Corollaire 1.2]{LiHong-Quan2006}. In this way, the LSI constant on $\widetilde{G}$ can at least be controlled by the LSI constant on $G$. Note that the optimal constant in the logarithmic Sobolev inequality with respect to the hypoelliptic heat kernel measure on the three-dimensional isotropic Heisenberg group is not known, even though a recent preprint \cite{AntonelliCalziGordina2025} proved sharp defective logarithmic Sobolev inequalities on $H$-type groups. 
\end{remark}

\bibliographystyle{amsplain}
\providecommand{\bysame}{\leavevmode\hbox to3em{\hrulefill}\thinspace}
\providecommand{\MR}{\relax\ifhmode\unskip\space\fi MR }
\providecommand{\MRhref}[2]{%
  \href{http://www.ams.org/mathscinet-getitem?mr=#1}{#2}
}
\providecommand{\href}[2]{#2}

\end{document}